\documentclass[a4paper,12pt]{article}
\usepackage{}
\usepackage{mathrsfs}
\usepackage{amssymb}
\usepackage{amsmath}
\usepackage{amsmath,amssymb,amsthm,graphics, graphicx,latexsym,amsfonts}
\usepackage{fancyhdr}
\usepackage{color}

\title{Chromatic-choosability of hypergraphs with high chromatic number\thanks{Supported by the
National Natural Science Foundation of China under Grant No.\,11471273 and 11561058.}}
\author
{
Wei Wang$^{\rm a,b}$,
Jianguo Qian$^{\rm a}$\thanks{Corresponding author: jgqian@xmu.edu.cn.}
\\
{\footnotesize$^{\rm a}$School of Mathematical Sciences, Xiamen University, Xiamen 361005, P. R. China}\\
{\footnotesize$^{\rm b}$College of Information Engineering, Tarim University, Alar 843300, P. R. China}
}
\usepackage{indentfirst}
\date{}
\begin{document}
\maketitle
\newtheorem{lem}{Lemma}[section]
\newtheorem{thm}[lem]{Theorem}
\newtheorem{prop}[lem]{Proposition}
\newtheorem{cor}[lem]{Corollary}
\newtheorem{conjecture}[lem]{Conjecture}
\newtheorem*{pf}{Proof}
\begin{abstract}
It was conjectured by Ohba and confirmed recently by Noel et al. that, for any graph $G$, if $|V(G)|\le 2\chi(G)+1$ then  $\chi_l(G)=\chi(G)$. This indicates that the graphs  with high chromatic number are chromatic-choosable. We show that this is also the case for uniform hypergraphs and further propose a generalized version of Ohba's conjecture: for any $r$-uniform hypergraph $H$ with $r\geq 2$, if $|V(H)|\le r\chi(H)+r-1$ then  $\chi_l(H)=\chi(H)$. We show that the condition of the proposed conjecture is sharp by giving two classes of $r$-uniform hypergraphs $H$ with $|V(H)|= r\chi(H)+r$ and $\chi_l(H)>\chi(H)$. To support the conjecture, we give two classes of  $r$-uniform hypergraphs $H$ with $|V(H)|= r\chi(H)+r-1$ and prove that $\chi_l(H)=\chi(H)$.
\end{abstract}


\section{Introduction}

For a graph or a hypergraph $G$, a  vertex coloring of $G$ is \emph{proper} if every edge contains a pair of vertices with different colors. For a positive integer $k$, a $k$-\emph{list assignment} of $G$ is a mapping $L$ which assigns to each vertex $v$ a set $L(v)$ of $k$ permissible colors.  Given a $k$-list assignment $L$, an \emph{$L$-coloring} of $G$ is a proper vertex coloring in which the color of every vertex $v$ is chosen from its list $L(v)$.  We say that $G$ is $L$-\emph{colorable} if $G$ has an $L$-coloring. A graph $G$ is called $k$-\emph{choosable} if for any $k$-list assignment $L$, $G$ is $L$-colorable. The \emph{list chromatic number} (or \emph{choice number}) $\chi_l(G)$ is the minimum $k$ for which $G$ is $k$-choosable. It is obvious that $\chi_l(G)\ge \chi(G)$, where $\chi(G)$ is the chromatic number of $G$. A graph $G$ is \emph{chromatic-choosable} if $\chi_l(G)= \chi(G)$. The notion of list coloring  was introduced independently by Vizing \cite{vizing1976} and by Erd\H{o}s, Rubin and Taylor \cite{erdos1979} initially for  ordinary graphs and then was extended to hypergraphs \cite{Alon2,Benzaken,Kahn,Ramamurthi,Saxton1,Saxton2}.

The list coloring for graphs has been extensively studied, much of the earlier fundamental work on which was surveyed in Alon \cite{Alon1}, Tuza  \cite{Tuza} and Kratochv\'{\i}l-Tuza-Voigt \cite{Kratochvil}. One direction of interests on list coloring focused on the estimation or asymptotic behaviour of the list chromatic number $\chi_l(G)$ compared to the degree of the vertices. In \cite{erdos1979}, Erd\H{o}s, Rubin and Taylor proved that the list chromatic number of the complete bipartite graph $K_{d,d}$ grows as binary logarithm of $d$ (the degree of $K_{d,d}$). More in general,  Alon \cite{Alon1} showed that the list chromatic number of any graph grows with the average degree. However, this is not the case for hypergraphs. It was shown that, when $r\geq 3$, it is not true in general that the list chromatic number of  $r$-uniform hypergraphs grows with its average degree \cite{Alon2}.  Even so, it was also shown that similar property holds for many classes of hypergraphs \cite{Alon2,Haxell,Saxton2}, including all the simple uniform hypergraphs (here, a hypergraph is \emph{simple} if different edges have at most one vertex in common) \cite{Saxton1}.

Another direction of interests on list coloring focused on the difference between the chromatic number $\chi(G)$ and list chromatic number $\chi_l(G)$. It was shown that  $\chi_l(G)$ can be much larger than  $\chi(G)$ for both the ordinary graphs \cite{erdos1979} and hypergraphs \cite{Haxell}. This yields a natural question: which graphs are chromatic-choosable?  A well known example concerning this question is the List Coloring Conjecture (attributed in particular to Vizing, see \cite{Haggkvist}), which says that every line-graph is chromatic-choosable. This conjecture was later extended to claw-free graphs \cite{gravier1}.

In addition to particular classes of the graphs that might be chromatic-choosable, the graphs with `high chromatic number' (compared to the number of the vertices in the graph) also received much attention. A trivial fact is that every complete graph is chromatic-choosable. In  \cite{ohba2002}, Ohba showed that, for any graph $G$, if $|V(G)|\le \chi(G)+\sqrt{2\chi(G)}$ then $\chi_l(G)=\chi(G)$. Further, in the same paper, Ohba  conjectured that if $|V(G)|\le 2\chi(G)+1$ then  $\chi_l(G)=\chi(G)$. This conjecture was recently confirmed by Noel, Reed and Wu \cite{noel2015}.

In this paper we focus on the chromatic-choosability of the uniform hypergraphs with high chromatic number, where the notion of uniform means that every edge consists of the same number of vertices. We show that the uni-form hypergraphs $H$ with high chromatic number are chromatic-choosable, that is, if $|V(H)|\le (r-\frac{1}{2})\chi(H)+\frac{r}{2}-1$ then $\chi_l(H)=\chi(H).$ Further, inspired by a recent Ohba-like conjecture for $d$-improper colorings given by Yan et al. \cite{yan2015} (See Conjecture \ref{imohba} below), we propose the following generalized version of Ohba's conjecture on $r$-uniform hypergraphs for any $r\geq 2$.
\begin{conjecture}\label{ourmain}
Let $r\ge 2$ and $H$ be an $r$-uniform hypergraph. If
$$|V(H)|\le r\chi(H)+r-1$$
then  $\chi_l(H)=\chi(H)$.
\end{conjecture}
It turns out that Conjecture \ref{ourmain} implies the conjecture of Yan et al. Furthermore, we show that the condition of Conjecture \ref{ourmain} is sharp by giving tow classes of $r$-uniform hypergraphs $H$ with $|V(H)|= r\chi(H)+r$ which are not chromatic-choosable. Finally, to support our conjecture we give two classes of  $r$-uniform hypergraphs $H$  with $|V(H)|=r\chi(H)+r-1$ and show that they are  chromatic-choosable.

\section{Chromatic-choosability with high\\ chromatic number}

For a graph $G$ and a set $C$ of colors, a coloring $f\colon\,V(G)\rightarrow C$ is a $d$-\emph{improper coloring} if each color class induces a subgraph with maximum degree at most $d$. Let  $\chi^d(G)$ and $\chi_l^d(G)$ denote the $d$-improper chromatic number and $d$-improper list chromatic number of $G$, respectively. Yan et al. \cite{yan2015} proposed an Ohba-like conjecture for $d$-improper colorings.
\begin{conjecture}\label{imohba}\cite{yan2015}
For any graph $G$, if $$|V(G)|\le (d+2)\chi^d(G)+(d+1)$$ then $\chi_l^d(G)=\chi^d(G).$
\end{conjecture}
For a graph $G$ and an integer $r\ge 2$, we construct an $r$-uniform hypergraph $G^{(r)}$ as follows:

\noindent 1). $V(G^{(r)})=V(G)$, and

\noindent 2). $E(G^{(r)})=\{S\subseteq V(G)\colon\, |S|=r~\text{and}~ \Delta(G[S])=r-1\}$, where $\Delta(G[S])$ is the maximum  degree of $G[S]$, i.e., the subgraph of $G$ induced by $S$.

\begin{prop}\label{corres}
For any graph $G$ and nonnegative integer $d$, we have $\chi(G^{(d+2)})=\chi^d(G)$ and $\chi_l(G^{(d+2)})=\chi^d_l(G)$.
\end{prop}
\begin{proof}
It is easy to see that a coloring $f$ of $G$ is $d$-improper if and only if $f$ is proper when regarded as a coloring of $G^{(d+2)}$. Thus, the assertion holds. \end{proof}

\begin{thm}
Conjecture \ref{ourmain} implies Conjecture \ref{imohba}.
\end{thm}
\begin{proof}
 Let $G$ be a graph with at most $(d+2)\chi^d(G)+(d+1)$ vertices. Let $r=d+2$ and $H=G^{(r)}$. By Proposition \ref{corres}, $\chi(H)=\chi^d(G)$. Note that $H$ and $G$ have the same vertex set. Thus, $|V(H)|\le (d+2)\chi^d(G)+(d+1)=r\chi(H)+r-1$. If Conjecture \ref{ourmain} is true, then $\chi_l(H)=\chi(H)$. This implies $\chi_l^d(G)=\chi^d(G)$ by Proposition \ref{corres}. The proof is completed.
\end{proof}

To support Conjecture \ref{imohba}, Yan et al. \cite{yan2015} also proved the following result.
\begin{thm} \label{impohbaweak}(Theorem 1, \cite{yan2015})
For any graph $G$ and integer $d\ge 0$, if
$$|V(G)|\le (d+1)\chi^d(G)+\sqrt{(d+1)\chi^d(G)}-d$$
then $\chi_l^d(G)=\chi^d(G).$
\end{thm}
For two vertex disjoint graphs $G_1$ and $G_2$, the \emph{join} of $G_1$ and $G_2$, denoted by $G_1 + G_2$, is obtained from their union by adding edges joining every vertex of $G_1$ to every vertex of $G_2$. Let $K_n$ be the complete graph with $n$ vertices.
\begin{cor}\label{impohbajoin}(Corollary 1, \cite{yan2015})
For any graph $G$ and integer $d\ge 0$, if $n\ge (|V(G)|+d)^2$ then $\chi_l^d(G+K_n)=\chi^d(G+K_n)$.
\end{cor}
Using Corollary \ref{impohbajoin}, Wang et al.  further improved Theorem \ref{impohbaweak} as follows.
\begin{thm}\label{weakerimproohba} (Theorem 2, \cite{wang2017})
For any graph $G$,  if $$|V(G)|\le (d+\frac{3}{2})\chi^d(G)+\frac{d}{2}$$ then $\chi_l^d(G)=\chi^d(G).$
\end{thm}

We remark that, all of Theorem \ref{impohbaweak},  Corollary \ref{impohbajoin} and Theorem \ref{weakerimproohba} have analogous forms for $r$-uniform hypergraphs by  properly extending the relevant concepts. For two vertex disjoint $r$-uniform hypergraphs $H_1$ and $H_2$, the \emph{join} of $H_1$ and $H_2$, denoted by $H_1+H_2$, is an $r$-uniform hypergraph with vertex set $V(H_1)\cup V(H_2)$ and edge set $E(H_1)\cup E(H_2)\cup \{S\subseteq V(H_1)\cup V(H_2)\colon\,|S|=r,S\not\subseteq V(H_1) \text{~and~} S\not\subseteq V(H_2)\}$. The \emph{complete $r$-uniform hypergraph} on $n$ vertices, denote by $K_n^{(r)}$,  has all $k$-subsets of its vertex set as edges.

Now we give  the following analogous form of Theorem  \ref{weakerimproohba} for $r$-uniform hypergraphs. The original proof for Theorem  \ref{weakerimproohba} is also valid for $r$-uniform hypergraphs by setting $r=d+2$ and replacing `$d$-improperly $L$-colorable', `graph join'  and `complete graph $K_n$' by `$L$-colorable', `$r$-uniform hypergraph join' and  `complete $r$-uniform graph $K_n^{(r)}$', respectively.
\begin{thm}\label{weakerhyperohba}
For any $r$-uniform hypergraph $H$, if
$$|V(H)|\le (r-\frac{1}{2})\chi(H)+\frac{r}{2}-1$$
 then $\chi_l(H)=\chi(H).$
\end{thm}

For $k$ positive integers $p_1,p_2,\ldots,p_k$, let $V_1,V_2,\ldots,V_k$ be $k$ disjoint sets of size $p_1$,$p_2$,$\ldots$,$p_k$, respectively. Following \cite{Emtander}, we define the $r$-complete $k$-partite hypergraph $K_{p_1,p_2,\ldots,p_k}^r$ with partite sets $V_1$,$V_2$,$\ldots$,$V_k$  as follows:

\noindent1). $V(K_{p_1,p_2,\ldots,p_k}^r)=V_1\cup V_2\cup\cdots\cup V_k$, and\\
2). $E(K_{p_1,p_2,\ldots,p_k}^r)=\{S\subseteq \bigcup_{i=1}^k V_i\colon\, |S|=r, S\not\subseteq V_i\text{~for\ any\ }i\in\{1,2,\ldots,k\}\}$.

We note that  the notion of $k$-partite hypergraph here means that each edge may contain two or more vertices from a partite set, which is different from others that used in some literatures. Nevertheless, when $r=2$, $K_{p_1,p_2,\ldots,p_k}^r$ agrees with  the usual complete $k$-partite graph $K_{p_1,p_2,\ldots,p_k}$. Further, if there are two $p_i$'s, say $p_1$ and $p_2$, which are less than $r-1$, then $K_{p_1,p_2,\ldots,p_k}^r$ is isomorphic to  $K_{p_1+1,p_2-1,\ldots,p_k}^r$ (or $K^r_{p_1+1,p_3,\ldots,p_k}$ if $p_2=1$). Therefore, in the following we always assume that $p_i\ge r-1$ for all $i\in \{1,2,\ldots,k\}$ with at most one exception. For simplicity, if $p_1=\cdots=p_s=p$ for some $s$ with $1\leq s\leq k$, we write $K_{p_1,p_2,\ldots,p_k}^r$ as $K_{p*s,p_{s+1},\ldots,p_k}^r$. Under this notation, $K_n^r$  contains no edges, which is different from the complete $r$-uniform hypergraph $K_n^{(r)}$ we defined ealier.
\begin{prop}\label{regular}
If $p_i\ge r-1$ for all $i\in \{1,2,\ldots,k\}$ with at most one exception, then $\chi(K_{p_1,p_2,\ldots,p_k}^r)=k$.
\end{prop}
\begin{proof}
Since $\chi(K_{p_1,p_2,\ldots,p_k}^r)\le k$ always holds, it suffices to show the reversed inequality. By the assumption of the proposition, $K_{(r-1)*(k-1),1}^r$ is a subgraph of $K_{p_1,p_2,\ldots,p_k}^r$ and therefore,
$\chi(K_{p_1,p_2,\ldots,p_k}^r)\ge \chi(K_{(r-1)*(k-1),1}^r)$. Further, notice that each $r$-subset of $V(K_{(r-1)*(k-1),1}^r)$ is an edge. We have
$$\chi(K_{(r-1)*(k-1),1}^r)\ge \left\lceil\frac{(r-1)(k-1)+1}{r-1}\right\rceil=k.$$
Therefore, $\chi(K_{p_1,p_2,\ldots,p_k}^r)\ge k$ and the proposition follows.
\end{proof}
It is easy to see that Conjecture \ref{ourmain} is true  if and only if it is true for all $r$-complete multipartite hypergraphs. Thus, in view of Proposition \ref{regular} we can restate  Conjecture \ref{ourmain} as follows.
\begin{conjecture}\label{regconj}
Let $r\ge 2$ and let $p_1,p_2,\ldots,p_k$ be $k$ positive integers such that $p_i\ge r-1$ for all $i\in \{1,2,\ldots,k\}$ with at most one exception. If $\sum\limits_{i=1}^{k} p_i\le rk+r-1$, then $\chi_l(K_{p_1,p_2,\ldots,p_k}^r)=k$.
\end{conjecture}
\section{Sharpness of Conjecture \ref{ourmain}}
It is well known that the condition of Ohba's Conjecture is sharp. Indeed, in \cite{enomoto2002} it was proved that the complete $k$-partite graph $G$ on $2k+2$ vertices is not chromatic-choosable if $k$ is even and either every part of $G$ has size 2 or 4, or every part of $G$ has size 1 or 3. In the following, we give an analogue of the former for $r$-uniform hypergraphs with $r\ge 3$ and a partial  generalization of the latter when $G=K_{3,3}$ to $r$-uniform hypergraphs with $r\ge 2$, indicating that the upper bound $r\chi(H)+r-1$ in Conjecture \ref{ourmain} is also sharp.
\begin{thm}\label{noncc}
For any integer $r\ge 3$, if  $k$ is a multiple of $r-1$ then
$$\chi_l(K_{2r,r*(k-1)}^r)>\chi(K_{2r,r*(k-1)}^r)=k.$$
\end{thm}
\begin{proof} Let $V_1,V_2,\ldots,V_k$ be the $k$ partite sets of $K_{2r,r*(k-1)}^r$, where
$$V_1=\{u_1,v_1,u_2,v_2,\ldots,u_r,v_r\}\ {\rm  and}\ V_i=\{w_{i,1},w_{i,2},\ldots,w_{i,r}\}$$
 for $i \in\{ 2,3,\ldots,k\}$. Let $C_1,C_2,\ldots,C_r$ be $r$ disjoint color sets of size $\frac{k}{r-1}$. Let $L$ be the $k$-list assignment of $K_{2r,r*(k-1)}^r$ defined by
$$L(w_{i,j})=L(u_j)=L(v_j)=\bigcup_{t=1,t\neq j}^{r} C_t,\ i=2,3,\ldots,k;\ j=1,2,\ldots,r.$$
We show that $K_{2r,r*(k-1)}^r$ is not $L$-colorable. Suppose to the contrary that $f\colon\,V\rightarrow \bigcup _{i=1}^{r} C_i$ is an $L$-coloring of $K_{2r,r*(k-1)}^r$, where $V=V_1\cup V_2\cup\cdots\cup V_k$. Define
$$S=\{c\in \bigcup_{j=1}^r C_j\colon\, f^{-1}(c)\not\subseteq V_1\} \text{~and~} T=\{c\in \bigcup_{j=1}^r C_j\colon\, f^{-1}(c)\subseteq V_1\}.$$
It is easy to check that $S\cup T$ is a bipartition of $\bigcup_{j=1}^r C_j$.

\noindent\textbf{Claim}: \emph{ $|f^{-1}(c)|\le r-1$ for each $c\in S$}.

Suppose to the contrary that $|f^{-1}(c)|>r-1$. Since $f^{-1}(c)\not\subseteq V_1$, there exists an  $r$-subset $W$ of $f^{-1}(c)$ such that $W\not\subseteq V_1$.  For any $i\in \{2,3,\ldots,k\}$, by the definition of $L$ we have $L(w_{i,1})\cap L(w_{i,2})\cap\cdots\cap L(w_{i,r})=\emptyset$. This means that $V_i$ has at least one vertex which is not assigned the color $c$ by $L$. Therefore, $W\neq V_i$, or equivalently, $W\not\subseteq V_i$ as $|W|=|V_i|$.  Combining with $W\not\subseteq V_1$, we have $W\not\subseteq V_i$ for all $i\in \{1,2,\ldots,k\}$. Thus, $W$ is an edge of $K_{2r,r*(k-1)}^r$. Further, since $f$ is a proper coloring, the edge $W$ is not monochromatic under $f$, which contradicts the fact that $W\subseteq f^{-1}(c)$.  This proves the claim.

Let $\ell=|\bigcup_{c\in T}f^{-1}(c)|$. Then $|\bigcup_{c\in S}f^{-1}(c)|=|V|-\ell=rk+r-\ell$. It follows from the above claim that $|S|\ge \lceil\frac{rk+r-\ell}{r-1}\rceil$. Since $\bigcap_{j=1}^rL(u_j)=\emptyset$ and $\bigcap_{j=1}^r L(v_j)=\emptyset$, any $2r-1$ vertices in $V_1$ share no common color in their lists. Thus, $|f^{-1}(c)|\le 2r-2$ for each $c\in T$ since $f^{-1}(c)\subseteq V_1$. Therefore, $|T|\ge \lceil\frac{\ell}{2r-2}\rceil$. Since $|\bigcup_{j=1}^r C_j|=\frac{rk}{r-1}$ and $S\cup T$ is a bipartition of $\bigcup_{j=1}^r C_j$, we have
$$\left\lceil\frac{rk+r-\ell}{r-1}\right\rceil+\left\lceil\frac{\ell}{2r-2}\right\rceil\le |S|+|T|\le \frac{rk}{r-1}.$$
As $k$ is a multiple of $r-1$, the above inequality can be reduced to
$$\left\lceil\frac{r-\ell}{r-1}\right\rceil+\left\lceil\frac{\ell}{2r-2}\right\rceil\le 0.$$

   On the other hand, notice that $\ell\le |V_1|=2r$. If $\ell\le 2r-1$ then
$$\left\lceil\frac{r-\ell}{r-1}\right\rceil+\left\lceil\frac{\ell}{2r-2}\right\rceil
 \ge  \frac{r-\ell}{r-1}+\frac{\ell}{2r-2}
   = \frac{2r-\ell}{2r-2}
  > 0,
$$
 a contradiction. If $\ell=2r$ then
$$
\left\lceil\frac{r-\ell}{r-1}\right\rceil+\left\lceil\frac{\ell}{2r-2}\right\rceil
  =   \left\lceil\frac{-r}{r-1}\right\rceil+\left\lceil\frac{2r}{2r-2}\right\rceil
  \geq -1+2
  > 0,
$$
where `$\geq$' holds as $r\ge 3$. This is again a contradiction and hence completes the proof of the theorem.\end{proof}

\begin{thm}\label{noncc2}
For any integer $r\ge 2$,
$$\chi_l(K_{(r+1)*r}^r)>\chi(K_{(r+1)*r}^r)=r.$$
\end{thm}
\begin{proof}
Let $H=K_{(r+1)*r}^r$ with $r$ partite sets $V_1,V_2,\ldots,V_r$, where $$V_i=\{v_{i,1},v_{i,2},\ldots,v_{i,r+1}\}{\rm{~for~}} i\in\{1,2,\ldots,r\}.$$ Let $L$ be the $r$-list assignment of $H$ defined by $L(v_{i,j})=\{1,2,\ldots,r+1\}\setminus \{j\}$ for $i\in \{1,2,\ldots,r\}$ and $j\in\{1,2,\ldots,r+1\}$. We show that $H$ is not $L$-colorable.

 Suppose to the contrary that $f\colon\, V(H)\rightarrow \{1,2,\ldots,r+1\}$ is an $L$-coloring of $H$. Then we have
 \begin{equation}\label{sumf}
 |f^{-1}(1)|+ |f^{-1}(2)|+\cdots+ |f^{-1}(r+1)|=|V(H)|=(r+1)r.
 \end{equation}
On the other hand, for any $i\in\{1,2,\ldots,r\}$, the lists of all $r+1$ vertices in $V_i$ have an empty intersection. Thus,  $|f^{-1}(k)|\le r$ for $k\in\{1,2,\ldots,r+1\}$. This, combining with (\ref{sumf}), implies that  $|f^{-1}(k)|= r$ for $k\in\{1,2,\ldots,r+1\}$. Therefore, for each $k\in \{1,2,\ldots, r+1\}$, $f^{-1}(k)$ must be contained in $V_i$ for some $i\in\{1,2,\ldots,r\}$ since otherwise  $f^{-1}(k)$ is an edge in $H$. By the pigeonhole principle, there exist two color classes, say $f^{-1}(1)$ and $f^{-1}(2)$, contained in the same partite set, say $V_1$. Consequently, $|f^{-1}(1)|+|f^{-1}(2)|=2r>r+1=|V_1|$. This is a contradiction and hence completes the proof.
\end{proof}
\section{Support for Conjecture \ref{ourmain}}
 We begin with some lemmas that are necessary for our forthcoming argument.

For an $r$-hypergraph $H$ and a subset $X\subseteq V(H)$, we denote by $H[X]$ the subgraph of $H$ induced by $X$, i.e., $H[X]=(X,\{e\colon\, e\in E(H), e\subseteq X\})$. For a list assignment $L$ of $H$, let $L(X)=\bigcup_{v\in X} L(v)$ and let $L_X$ denote $L$ restricted to $X$. We may omit the subscript of $L_X$ when there is no ambiguity. For example, when  $H[X]$ is $L_X$-colorable we simply say that $H[X]$ is $L$-colorable. For a color set $C$, let $L\setminus C$ be the list assignment of $H$ defined by $(L\setminus C)(v)=L(v)\setminus C$ for each vertex $v\in V(H)$.
\begin{lem}\label{comb}
 Let $X\cup Y=V(H)$ be a bipartition of the vertex set of an $r$-uniform hypergraph $H$ and $f$ be an $L$-coloring of $H[X]$. If there is a color set $C$ such that $C\supseteq f(X)$ and $H[Y]$ is $L\setminus C$-colorable, then $H$ is $L$-colorable.
\end{lem}
\begin{proof}
Let $g$ be an $L\setminus C$-coloring of $H[Y]$. Define a coloring $h$ of $H$ by $h(v)=f(v)$ if $v\in X$, and $h(v)=g(v)$ if $v\in Y$. One can easily check that $h$ is an $L$-coloring of $H$.
\end{proof}
\begin{lem}\label{ghall}
Let $L$ be a list assignment of an $r$-uniform hypergraph $H$. If $(r-1)|L(X)|\ge |X|$ for each nonempty subsets $X\subseteq V(H)$, then $H$ is $L$-colorable.
\end{lem}
\begin{proof}
Consider the bipartite graph $B$ with vertex partition $V(B)=(V(H),$ $C)$, where $C$ consists of $(r-1)$ copies of $L(V(H))$ and, for each $v\in V(H)$, $v$ is adjacent to the $(r-1)$ copies of $L(v)$. Clearly, for each $X\subseteq V(H)$, we have $N_B(X)=(r-1)L(X)$ and hence $|N_B(X)|\ge |X|$ by the condition of the lemma. Thus, by Hall's Matching Theorem, there exists a matching $M$ that saturates $V(H)$. We associate $M$ with an $L$-coloring $f_M$ of $H$ defined by $f_M(v)=c_M(v)$ for any $v\in V(H)$, where $c_M(v)$ is the color matched to $v$ by $M$. We can see that each vertex $v$ is colored by a color from its own list $L(v)$, and each color class of $H$ induced by $f_M$ contains at most $r-1$ vertices. This means that each edge of $H$ contains at least two vertices with different colors since $H$ is $r$-uniform.  Thus, $f_M$ is proper and therefore, $H$ is $L$-colorable.
\end{proof}
The next two lemmas are the extensions of two methods which are frequently used in the study of the list colorings for ordinary graphs. Our proof follows the techniques given in \cite{Kierstead,Reed},  with only slight modifications.
\begin{lem}\label{eitheror}
For a list assignment $L$ of an $r$-uniform hypergraph $H$, if $H[X]$ is $L$-colorable for each nonempty subset $X\subseteq V(H)$ with $(r-1)|L(X)|< |X|$, then  $H$ is $L$-colorable.
\end{lem}
\begin{proof}
If  $(r-1)|L(X)|\ge |X|$ for each nonempty subset $X\subseteq V(H)$, then we are done by Lemma \ref{ghall}. We now assume that $X$ is a maximal nonempty subset of  $V(H)$ such that $(r-1)|L(X)|< |X|$.  Let $C=L(X),Y=V(H)\setminus X$ and let $S$ be an arbitrary nonempty subset of $Y$. Then by the maximality of $X$, $(r-1)|L(X\cup S)|\geq |X\cup S|$. On the other hand, notice that  $|L(X\cup S)|=|L(X)|+|(L\setminus C)(S)|$  and $|X\cup S|=|X|+|S|$ as $X\cap S=\emptyset$. So we have $(r-1)|(L\setminus C)(S)|\ge |S|$. Consequently,  $H[Y]$ is $L\setminus C$-colorable by Lemma \ref{ghall}. Let $f$ be any $L$-coloring of $H[X]$. Clearly, $L(X)\supseteq f(X)$, that is, $C\supseteq f(X)$. Therefore, $H$ is $L$-colorable by Lemma \ref{comb}.
\end{proof}
\begin{lem}\label{small}
An $r$-uniform hypergraph $H$ is $k$-choosable if $H$ is $L$-colorable for every $k$-list assignment $L$ such that $(r-1)|L(V(H))|<|V(H)|$.
\end{lem}
\begin{proof}
Let $L$ be an arbitrary $k$-list assignment of $H$. Let $X$ be a nonempty subset $X\subseteq V(H)$ such that
$(r-1)|L(X)|< |X|$ and let $x\in X$. Define the list assignment $L'$ of $H$ by $L'(v)=L(v)$ if $v\in X$ and $L'(v)=L(x)$ otherwise. Clearly,  $L'(V(H))=L(X)$ and hence $(r-1)|L'(V(H))|<|X|\le |V(H)|$. Thus by the condition of the lemma, $H$ is $L'$-colorable and so is $H[X]$. Since $L'_X=L_X$, $H[X]$ is  $L$-colorable. So by Lemma \ref{eitheror}, $H$ is $L$-colorable.
\end{proof}

Gravier and Maffray \cite{gravier} showed that $K_{3,2*(k-1)}$ is chromatic-choosable, which gave a support to Ohba's conjecture before the conjecture was proved. The following two theorems are the generalizations of this result to uniform hypergraphs and therefore, give a support to Conjecture \ref{ourmain}.

For a color $c$ of $L$ and a vertex subset $X$ of $H$, the {\it multiplicity} of $c$  in $X$ is defined by $|\{v:v\in X,c\in L(v)\}|$, that is, the total times of $c$ that appears in the lists of the vertices in $X$. For a list assignment $L$, the multiplicity of $c$ in $X$ is denoted by $\eta_{L,X}(c)$, or simply $\eta_X(c)$ when the list assignment is clear.
\begin{thm}\label{cc2rm1}
$\chi_l(K_{2r-1,r*(k-1)}^r)=k$ for $r\ge 2$ and $k\ge 1$.
\end{thm}
\begin{proof} We prove it by contradiction. Suppose $k$ is the minimal positive integer such that $K_{2r-1,r*(k-1)}^r$ is not $k$-choosable. Note that if $k=1$ then $K_{2r-1,r*(k-1)}^r$ contains no edges and therefore is trivially $1$-choosable. Thus $k\ge 2$. Write $H= K_{2r-1,r*(k-1)}^r$. Since $H$ is not $k$-choosable, Lemma \ref{small} implies that there exists a $k$-list assignment $L$ such that $(r-1)|L(V(H))|<|V(H)|$ and $H$ is not $L$-colorable. Let  $V_1,V_2,\ldots,V_k$ be all partite sets of $H$, where $|V_1|=2r-1$ and $|V_i|=r$ for $i\in\{2,3,\ldots,k\}$. As $(r-1)|L(V(H))|<|V(H)|=rk+r-1$ and $L(V_i)\subseteq L(V(H))$, we have $(r-1)|L(V_i)|\le  rk+r-2$ and hence
\begin{equation}\label{boundonLV}
|L(V_i)|\le \left\lfloor\frac{rk+r-2}{r-1}\right\rfloor,\ i=1,2,\ldots,k.
\end{equation}

\noindent\textbf{Claim 1}: \emph{$\bigcap_{v\in V_i}L(v)=\emptyset$ for each $i\in \{2,3,\ldots,k\}$.}

Suppose to the contrary that there exists a color $c^*\in \bigcap_{v\in V_i}L(v)$ for some $i\in\{2,3,\ldots,k\}$. We use $c^*$ to color all vertices in $V_i$ and let $Y=V(H)\setminus V_i$. Note that $H[Y]=K_{2r-1,r*(k-2)}^r$. By the minimality of $k$, $H[Y]$  is $(k-1)$-choosable. Therefore, $H[Y]$ is $L\setminus\{c^*\}$-colorable since $(L\setminus\{c^*\})(v)$ contains at least $k-1$ colors for each $v\in Y$. So by Lemma \ref{comb}, $H$ is $L$-colorable. This is a contradiction and hence  Claim 1 follows.

Let
\begin{equation}
\xi=\left\lceil\frac{(2r-1)k}{\left\lfloor\frac{rk+r-2}{r-1}\right \rfloor}\right\rceil.
\end{equation}

\noindent\textbf{Claim 2}: \emph{$L$ has a color $\bar{c}$ such that $\eta_{V_1}(\bar{c})\ge \xi$}.

Clearly, $\sum_{c\in L(V_1)}\eta_{V_1}(c)=\sum_{v\in V_1}|L(v)|=(2r-1)k$.  Let $\bar{c}$ be the color such that $\eta_{V_1}(\bar{c})$ is maximum. By (\ref{boundonLV})  we have
$$\eta_{V_1}(\bar{c})
\ge\frac{\sum_{v\in V_1}|L(v)|}{|L(V_1)|}
\ge\frac{(2r-1)k}{\lfloor\frac{rk+r-2}{r-1}\rfloor},
$$
which implies  $\eta_{V_1}(\bar{c})\ge \xi$. Thus,  Claim 2 follows.

\noindent\textbf{Claim 3}:  \emph{$|L(V_i)|= \left\lfloor\frac{rk+r-2}{r-1}\right\rfloor$ for each $i\in \{2,3,\ldots,k\}$.}

Suppose to the contrary that  $|L(V_i)|\neq\left\lfloor\frac{rk+r-2}{r-1}\right\rfloor$ for some $i\in\{2,3,\ldots,k\}$. Then, by (\ref{boundonLV}), $|L(V_i)|\le\left\lfloor\frac{rk+r-2}{r-1}\right\rfloor-1$ and hence $|L(V_i)|\le\frac{rk-1}{r-1}$.  Let $c_i$ be the color in $L(V_i)$ such that $\eta_{V_i}(c_i)$ is maximum. By an argument similar to the proof of  Claim 2, we have
$$\eta_{V_i}(c_i)
 \ge  \frac{\sum_{c\in L(V_i)}\eta_{V_i}(c)}{|L(V_i)|}
  =  \frac{rk}{|L(V_i)|}
  \ge \frac{rk}{\frac{rk-1}{r-1}}
 > r-1.
$$
This means that all vertices in $V_i$ have a common color in their lists. This contradicts Claim 1 and therefore, Claim 3 follows.

\noindent\textbf{Claim 4}: $\xi\ge r+ (r-1)\left(\frac{rk+r-2}{r-1}-\left\lfloor\frac{rk+r-2}{r-1}\right\rfloor\right)$ and in particular, $\xi\ge r$.

Write $k-1=(r-1)p+q$, where $p=\lfloor\frac{k-1}{r-1}\rfloor$ and $0\le q\le r-2$. Then $\frac{rk+r-2}{r-1}=k+1+\frac{k-1}{r-1}=rp+q+2+\frac{q}{r-1}$. Thus, the first inequality in the claim is reduced to
\begin{equation}
\left\lceil\frac{(2r-1)((r-1)p+q+1)}{rp+q+2}\right\rceil\ge r+q,
\end{equation}
that is,
 \begin{equation}
(2r-1)((r-1)p+q+1)>(rp+q+2)(r+q-1).
\end{equation}

Let $\Delta=(2r-1)((r-1)p+q+1)-(rp+q+2)(r+q-1)=-q^2+(r-2-pr)q+(1+p-2pr+pr^2)$. In order to show $\Delta>0$ we consider the quadratic function $f(x)=-x^2+(r-2-pr)x+(1+p-2pr+pr^2)$. Note that $0\le q\le r-2$ and $\Delta=f(q)$.   As $f(x)$ is strictly concave on the interval $[0,r-2]$, the minimum value of $f(x)$ must be attained at $x=0$ or $r-2$. Direct calculation leads to
$f(0)=1+p-2pr+pr^2=rp(r-2)+p+1>0$ and $f(r-2)=p+1>0$. Therefore, $f(x)>0$ on $[0,r-2]$ and  hence $\Delta>0$. This proves Claim 4.

Let $X=\{v\in V_1\colon\,\bar{c}\in L(v)\}$.  Let $Y=V(H)\setminus X$, $V_1'=V_1\setminus X$ and $L'=L_Y\setminus\{\bar{c}\}$. Then by Claims 2 and  4,  we have $|X|\ge \xi\geq r$ and therefore,
\begin{equation}\label{V1p}
|V_1'|=|V_1\setminus X|\le 2r-1-\xi\le r-1.
\end{equation}
Clearly, $|L'(v)|=|L(v)|=k$ for each $v\in V_1'$,  and $|L'(v)|\ge |L(v)|-1=k-1$ for each $v\in V_i$, $i\in\{2,3,\ldots,k\}$.

\noindent\textbf{Claim 5}:  \emph{$H[Y]$ is $L'$-colorable}.

Let $S$ be an arbitrary nonempty subset of $Y$. By Lemma \ref{ghall}, it suffices to show that $(r-1)|L'(S)|\ge |S|$. To this end, we consider two cases.

\noindent\emph{Case 1}:  $V_i\not\subseteq S$ for any $i\in \{2,3,\ldots,k\}$.

In this case, we have
 \begin{equation}\label{capp}
 |S\cap (V(H)\setminus V_1)|\le (|V_2|-1)+\cdots+(|V_k|-1)=(r-1)(k-1).
\end{equation}\label{cap}
 Notice that $|L'(S)|\ge |L'(v)|\ge k-1$ for any vertex $v$ in $S$. So by (\ref{capp}), if $S\cap V_1'=\emptyset$ then $(r-1)|L'(S)|\ge(r-1)(k-1)|\geq |S\cap (V(H)\setminus V_1)|=|S|$, as desired. Now we assume that $S\cap V_1'\not=\emptyset$.
Then by  (\ref{V1p}) and (\ref{capp}) we have $|S|=|(S\cap V'_1)\cup (S\cap (V(H)\setminus V_1))|\le (r-1)+(r-1)(k-1)=(r-1)k$. Let $v\in S\cap V_1'$. Then $|L'(v)|=k$ and hence $|L'(S)|\ge k$. Again we have $(r-1)|L'(S)|\ge|S|$.

\noindent\emph{Case 2}: $V_i\subseteq S$ for some $i\in\{2,3,\ldots,k\}$.

By Claim 3, $L'(V_i)\ge  \left\lfloor\frac{rk+r-2}{r-1}\right\rfloor-1$.   On the other hand, by the first inequality in (\ref{V1p}),
$|S|\le |V_1'|+|V_2|+\cdots+|V_k|\le 2r-1-\xi+r(k-1)$. Therefore, by Claim 4,
\begin{eqnarray*}\label{diffPlP}
(r-1)|L'(S)|-|S|&\ge&(r-1)\left(\left\lfloor\frac{rk+r-2}{r-1}\right\rfloor-1\right)-(2r-1-\xi+r(k-1))\\
&=&\xi+(r-1)\left\lfloor\frac{rk+r-2}{r-1}\right\rfloor-rk-2r+2\\
&\ge& r+ (r-1)\left(\frac{rk+r-2}{r-1}-\left\lfloor\frac{rk+r-2}{r-1}\right\rfloor\right)\\
& &+(r-1)\left\lfloor\frac{rk+r-2}{r-1}\right\rfloor-rk-2r+2\\
&=&0.
  \end{eqnarray*}
Thus, $(r-1)|L'(S)|\ge|S|$, as desired.

From the above two cases,  Claim 5 follows.

Finally, by Claim 5 and Lemma \ref{comb}, $H$ is $L$-colorable. This is a contradiction and hence completes the proof of this theorem.\end{proof}
\begin{cor}\label{detnoncc}
For any integer $r\ge 3$, if  $k$ is a multiple of $r-1$ then
$$\chi_l(K_{2r,r*(k-1)}^r)=k+1.$$
\end{cor}
\begin{proof}
By Theorem \ref{noncc}, $\chi_l(K_{2r,r*(k-1)}^r)>k.$ Clearly, $\chi_l(K_{2r,r*(k-1)}^r)\le 1+\chi_l(K_{2r-1,r*(k-1)}^r)$. Thus, by Theorem \ref{cc2rm1} $\chi_l(K_{2r,r*(k-1)}^r)\le k+1$. This proves the corollary.
\end{proof}
The following result gives the second generalization of $K_{3,2*(k-1)}$ for supporting our conjecture.
\begin{thm}\label{ccrp1}
$\chi_l(K_{(r+1)*(r-1),r*(k-r+1)}^r)=k$ for $r\ge 2$ and $k\ge r-1$.
\end{thm}
Before proving,  we need first to show that $\chi_l(K_{(r+1)*(r-1)}^r)=r-1$. In fact, we prove the following more general result.
\begin{prop}\label{start}
$\chi_l(K_{(r+1)*k}^r)=k$  for $r\ge 2$ and $k\le r-1$.
\end{prop}
\begin{proof}
If  $r=2$ then $k=1$ and the assertion  trivially holds. We may assume that $r\ge 3$. We prove the proposition by induction on $k$. Since $\chi_l(K_{(r+1)*k}^r)\ge\chi(K_{(r+1)*k}^r)=k$, it suffices to show that $K_{(r+1)*k}^r$ is $k$-choosable. If $k=1$ then $K_{(r+1)*k}^r$ contains no edges and hence is $1$-choosable. Let $1<k\le r-1$ and assume that $K_{(r+1)*t}^r$ is $t$-choosable for any $t<k$. For simplicity, let $H=K_{(r+1)*k}^r$ and let  $V_1,V_2,\ldots,V_{k}$ be the $k$ partite sets of $H$. We need to show that $H$ is $k$-choosable.

Let $L$ be any $k$-list assignment of $H$ such that
 \begin{equation}\label{basicinequ}
 (r-1)|L(V(H))|<|V(H)|=(r+1)k.
  \end{equation}
 By Lemma \ref{small}, to show that $H$ is $k$-choosable, it suffices to show that $H$ is $L$-colorable. If there is some $V_i$ such that all vertices in $V_i$ have a common color $c^*$ in their lists, then we can color each vertex in $V_i$ by $c^*$ and remove $c^*$ from the lists of all other vertices in $H$. Using induction on $k$ and Lemma \ref{comb}, one can easily verify that $H$ is $L$-colorable.

In the following, we assume that $\bigcap_{v\in V_i}L(v)=\emptyset$ for any $i\in\{1,2,\ldots,k\}$. As $|V_i|=r+1$ we have $\eta_{V_i}(c)\le r$ for each $c\in L(V_i)$. For  each $i\in\{1,2,\ldots,k\}$, let $C_i=\{c\in L(V_i)\colon\,\eta_{V_i}(c)=r\}$.  Thus, for each color $c\in L(V_i)\setminus C_i$, we have $\eta_{V_i}(c)\le r-1$ and hence,
\begin{equation}\label{Cibound}
r|C_i| +(r-1)(|L(V_i)|-|C_i|)\ge \sum_{v\in V_i}|L(v)|=(r+1)k
\end{equation}
for each $i\in\{1,2,\ldots,k\}$. Equivalently, $|C_i|\ge (r+1)k-(r-1)|L(V_i)|$. Since $|L(V_i)|\le |L(V(H))|$, we have $|C_i|>0$ by (\ref{basicinequ}).

Let $I$ be a maximal subset of $\{1,2,\ldots,k\}$ such that $\{C_i\colon\,i\in I\}$ has a system of distinct representatives and let $s=|I|$. Since $C_i$ is nonempty, $s\ge 1$. With no loss of generality, we may assume that $I=\{1,2,\ldots,s\}$. Let $(c_1,c_2,\ldots,c_s)$ be a system of distinct representatives of $(C_1,C_2,\ldots,C_s)$. Notice that $\eta_{V_i}(c_i)=r$ and $|V_i|=r+1$. For each $i\in \{1,2,\ldots,s\}$, let $v_i$ be the only vertex in $V_i$ such that $c_i\not\in L(v_i)$. Let $H'=H[\{v_1,\ldots,v_s\}\cup V_{s+1}\cup\cdots \cup V_{k}]$ and define a list assignment $L'$ on the hypergraph $H'$ by $L'(v)=L(v)\setminus \{c_1,\ldots,c_s\}$ for any $v\in V(H')$. For each $i\in \{1,2,\ldots,s\}$, we use $c_i$ to color all vertices in $V_i$ except $v_i$. By Lemma \ref{comb}, to show that $H$ is $L$-colorable, it suffices to show that $H'$ is $L'$-colorable.

For each $i\in\{1,2,\ldots,s\}$, as $s\le k$ and $c_i\not\in L(v_i)$, we have $|L'(v_i)|\ge L(v_i)-(s-1)=k-(s-1)\ge 1$. If $s=k$ then $|V(H')|=k<r$ and hence $H'$ contains no edges. In this case, $H'$ is trivially $L'$-colorable. Thus, we assume that $s\le k-1$. For each $p\in \{s+1,s+2,\ldots,k\}$, by the maximality of $I$, we have
$C_p\subseteq\{c_1,c_2,\ldots,c_s\}$ and hence $|C_p|\le s$.

Let $S$ be an arbitrary subset of $V(H')$. We consider three cases:

\noindent\emph{Case 1}: $v_i\not\in S$ for any $i\in \{1,2,\ldots,s\}$.

In this case, $H'[S]$ is an induced subgraph of $K_{(r+1)*(k-s)}^r$. Further, by the induction hypothesis, $K_{(r+1)*(k-s)}^r$ is $(k-s)$-choosable. Therefore, $H'[S]$ is $(k-s)$-choosable.  As $|L'(v)|\ge |L(v)|-s=k-s$ for each $v\in S$, $H'[S]$ is $L'$-colorable.

\noindent\emph{Case 2}: $v_i\in S$ for some $i\in \{1,2,\ldots,s\}$ and $V_p\not \subseteq S$ for any $p\in\{s+1,$ $s+2,\ldots,k\}$.

In this case, $|S|\le r(k-s)+s$. As $|L'(v_i)|\ge k-s+1$ and $k\le r-1$, we have
$$(r-1)|L'(S)|-|S|\ge (r-1)(k-s+1)-(r(k-s)+s)=r-1-k\ge 0,$$
that is, $(r-1)|L'(S)|\ge|S|$.

\noindent\emph{Case 3}: $v_i\in S$ for some $i\in \{1,2,\ldots,s\}$ and $V_p \subseteq S$ for some $p\in \{s+1,$ $s+2,\ldots,k\}$.

By (\ref{Cibound}), we have $(r-1)|L(V_p)|\ge (r+1)k-|C_p|$ and hence $(r-1)|L(V_p)|\ge (r+1)k-s$ as $|C_p|\le s$. Therefore,
$$(r-1)|L'(S)|\ge (r-1)|L'(V_p)|\ge (r-1)(|L(V_p)|-s)\ge (r+1)k-rs.$$
On the other hand, $|S|\le |V(H')|=(r+1)k-rs$. Thus, $(r-1)|L'(S)|\ge |S|$.

By the above three cases, for any $S\subseteq V(H')$, either $(r-1)|L'(S)|\ge|S|$ or $H'[S]$ is $L'$-colorable. It follows from Lemma \ref{eitheror} that $H'$ is $L'$-colorable. Thus, $H$ is $L$-colorable and hence $k$-choosable. This proves the proposition by induction.
\end{proof}

\noindent\emph{Proof of Theorem \ref{ccrp1}}.  We prove the theorem by induction on $k$. If  $k= r-1$ the the assertion holds by Proposition \ref{start}.  Now let $k\ge r$ and assume that $K^r_{(r+1)*(r-1),r*(k-r)}$ is $(k-1)$-choosable. We are going to show that $K^r_{(r+1)*(r-1),r*(k-r+1)}$ is $k$-choosable. Write $H=K^r_{(r+1)*(r-1),r*(k-r+1)}$ and let  $V_1,V_2,\ldots,V_{k}$ be the partite sets of $H$ with $|V_i|=r+1$ for $i\in\{1,2,\ldots,r-1\}$ and $|V_i|=r$ for $i\in\{r,r+1,\ldots,k\}$.

Let $L$ be any $k$-list assignment of $H$ such that
\begin{equation}\label{basicinequ2}
(r-1)|L(V(H))|<|V(H)|=rk+r-1.
\end{equation}
By Lemma \ref{small}, it suffices to show that $H$ is $L$-colorable.

For some $i\in \{1,2,\ldots,k\}$, if  all vertices in $V_i$ have a common color, say $c^*$, in their lists, then we can color each vertex in $V_i$ by $c^*$. Let $H'$ be the subgraph of $H$ induced by $V(H)\setminus V_i$. That is, $H'=K_{(r+1)*(r-2),r*(k-r+1)}^r$ if $i\leq r-1$ or $K_{(r+1)*(r-1),r*(k-r)}^r$ if $i>r-1$, both of which are subgraphs of $K_{(r+1)*(r-1),r*(k-r)}^r$. Further, by the induction hypothesis, $K_{(r+1)*(r-1),r*(k-r)}^r$ is $(k-1)$-choosable and so is $H'$. Let $L'$ be the list assignment of $H'$ defined by $L'(v)=L(v)\setminus \{c^*\}$ for each $v\in V(H')$. Then $|L'(v)|\ge k-1$ and hence $H'$ is $L'$-colorable. Thus, $H$ is $L$-colorable by Lemma \ref{comb}.

We now assume  that $\bigcap_{v\in V_i}L(v)=\emptyset$ for each $i\in \{1,2,\ldots,k\}$. The following discussion is much similar to the proof of Proposition \ref{start}. For each $i\in \{1,2,\ldots,r-1\}$  let $C_i=\{c\in L(V_i)\colon\,\eta_{V_i}(c)=r\}$.  Then (\ref{Cibound}) holds for each $i\in\{1,2,\ldots,r-1\}$ and, therefore, $|C_i|\ge (r+1)k-(r-1)|L(V_i)|$. Since $|L(V_i)|\leq |L(V(H))|$ and $k\ge r$, it follows by (\ref{basicinequ2}) that $|C_i|> 1.$

Let $I$ be a maximal subset of $\{1,2,\ldots,r-1\}$ such that $\{C_i\colon\,i\in I\}$ has a system of distinct representatives, and let $s=|I|$. It is clear that $1\leq s\leq r-1$ as  $C_i\neq\emptyset$ for each $i\in \{1,2,\ldots,r-1\}$. With no loss of generality, we assume that $I=\{1,2,\ldots,s\}$ and $(c_1,c_2,\ldots,c_s)$ is a system of distinct representatives of $(C_1,C_2,\ldots,C_s)$. For each $i\in\{1,2,\ldots,s\}$, let $v_i$ be the only vertex of $V_i$ such that $c_i\not\in L(v_i)$. Let $H'=H[\{v_1,\ldots,v_s\}\cup V_{s+1}\cup\cdots\cup V_{k}]$ and define $L'(v)=L(v)\setminus \{c_1,c_2\ldots,c_s\}$ for every $v\in V(H')$. It suffices to show that $H'$ is $L'$-colorable by Lemma \ref{comb}.

  For each $i\in \{1,2,\ldots,s\}$, since $c_i\not\in L(v_i)$, we have
  \begin{equation}\label{singbound}
 |L'(v_i)|\ge |L(v_i)|-(s-1)= k-s+1.
 \end{equation}
 For each $p\in \{r,r+1,\ldots,k\}$, since $\bigcap_{v\in V_p}L(v)=\emptyset$,  each color of $L(V_p)$ appears at most $r-1$ times in $V_p$. Therefore,
 \begin{equation}\label{lvp}
|L(V_p)|\ge \frac{\sum_{v\in V_p}|L(v)|}{r-1}=\frac{rk}{r-1}.
\end{equation}
As $|L'(V_p)|\ge (|L(V_p)|-s)$, (\ref{lvp}) implies
\begin{equation}\label{rpartbound}
(r-1)|L'(V_p)|\ge rk-(r-1)s.
\end{equation}
If $s<r-1$, then for each $q\in\{s+1,s+2,\cdots,r-1\}$, we have $C_q\subseteq \{c_1,\ldots,c_s\}$ by the maximality of $I$. Thus $|C_q|\le s$. It follows from (\ref{Cibound}) (regard $i$ as $q$) that $(r-1)|L(V_q)|\ge (r+1)k-|C_q|\ge (r+1)k-s$. Thus,
\begin{equation}\label{midbound}
(r-1)|L'(V_q)|\ge (r-1)(|L(V_q)|-s)\geq (r+1)k-rs.
\end{equation}

Let $S$ be an arbitrary subset of $V(H')$. We will show that either $H'[S]$ is $L'$-colorable or $(r-1)|L(S)|\ge |S|$.

First assume that $s<r-1$ and $V_q\subseteq S$ for some $q\in \{s+1,s+2,\ldots,r-1\}$. Note that $|S|\le |V(H')|=rk+(r-1)-rs$, $|L'(S)|\ge |L'(V_q)|$ and $k\ge r$. It follows from (\ref{midbound}) that
\begin{equation}\label{midbound2}
(r-1)|L'(S)|\ge (r+1)k-rs\ge rk+r-rs> |S|,
\end{equation}
as desired. In the following, we always assume that $V_q\not\subseteq S$ for any $q\in \{s+1,s+2,\ldots,r-1\}$, unless $s=r-1$. Under this assumption, we have $|S\cap V_i|\le r$ for all $i\in \{s+1,s+2,\ldots,k\}$. We consider three cases:

\noindent\emph{Case 1}: $v_i\not\in S$ for any $i\in\{ 1,2,\ldots,s\}$.

In this case, $H'[S]$ is an induced subgraph of $K_{r*(k-s)}^r$ and hence of $K_{2r-1,r*(k-s-1)}^r$. Thus, $H'[S]$ is $(k-s)$-choosable by Theorem \ref{cc2rm1}. Since $|L'(v)|\ge |L(v)|-s= k-s$ for each $v\in S$, $H'[S]$ is $L'$-colorable, as desired.

\noindent\emph{Case 2}: $v_i\in S$ for some $i\in\{ 1,2,\ldots,s\}$ and $V_p\not\subseteq S$ for any $p\in\{r,$ $r+1,\ldots,k\}$.

Combining with our assumption that $V_q\not\subseteq S$ for $q\in\{s+1,s+2,\ldots,r-1\}$, we have
$V_j\not\subseteq S$ for any $j\in \{s+1,s+2,\ldots,k\}$. Thus,
 $$S\le |V(H')|-(k-s)=(rk+(r-1)-rs)-(k-s)=(r-1)(k+1-s).$$
 As $v_i\in S$, we have $|L'(S)|\ge |L'(v_i)|$, implying that $|L'(S)|\ge k+1-s$ by (\ref{singbound}). Thus, $ (r-1)|L'(S)|\ge|S|$.

 \noindent\emph{Case 3}: $v_i\in S$ for some $i\in\{ 1,2,\ldots,s\}$ and $V_p\subseteq S$ for some $p\in\{r,$ $r+1,\ldots,k\}$.

 In this case, again by our assumption that $V_p\not\subseteq S$ for any $p\in\{s+1,$\ \  $s+2,\ldots,r-1\}$, we have $|S|\le |V(H')|-(r-1-s)=(rk+$ $(r-1)-rs)-(r-1-s)=rk-(r-1)s$. Since $V_p\subseteq S$, so by  (\ref{rpartbound}) we have $(r-1)|L'(S)|\ge (r-1)|L'(V_p)|\ge rk-(r-1)s\geq |S|$.

 By the above three cases, for any $S\subseteq V(H')$, either $(r-1)|L'(S)|\ge |S|$ or $H'[S]$ is $L'$-colorable. Therefore, $H'$ is $L'$-colorable by Lemma \ref{eitheror}. This completes the proof of Theorem \ref{ccrp1}.
\begin{cor}\label{detnoncc2}
$\chi_l(K_{(r+1)*r}^r)=r+1$ for  $r\ge 2$.
\end{cor}
\begin{proof}
By Theorem \ref{noncc2}, $\chi_l(K_{(r+1)*r}^r)\ge r+1.$ On the other hand, using Theorem \ref{ccrp1} for $k=r$, we have $\chi_l(K_{(r+1)*(r-1),r}^r)=r.$ Thus
$\chi_l(K_{(r+1)*r}^r)\le r+1.$ This proves the corollary.
\end{proof}
 
\end{document}